\documentclass{article}

\usepackage[numbers]{natbib}
\usepackage{blindtext}
\usepackage{lineno,hyperref}
\modulolinenumbers[5]
\usepackage{graphicx} 
\usepackage{subfigure}
\usepackage{tikz}
\usepackage{multirow}
\usetikzlibrary{positioning}
\usepackage{colortbl}
\usepackage{amsmath,amssymb}  
\usepackage{chngcntr}
\usepackage{epstopdf}
\usepackage{lipsum}   
\usepackage{setspace}
\usepackage{nameref}
\usepackage{float}
\usepackage{stackengine}
\usepackage{makecell}
 \usepackage{eurosym}
\usetikzlibrary{bayesnet}
\usepackage{natbib}
\usepackage{xurl}
\usepackage{adjustbox} 
\usepackage{tikz}
\usepackage{pgfplotstable,colortbl}
\usepackage{xcolor}
\pgfplotsset{compat=1.12}
\usepgfplotslibrary[colormaps] 
\usetikzlibrary{pgfplots.colormaps} 
\usepackage{pgfplotstable}
\usepackage{pgfplots}
\usepackage{authblk}
\usepackage[T1]{fontenc}
\usepackage[utf8]{inputenc}
\usepackage{lmodern} 
\usepackage{lineno}
\pgfplotsset{compat=1.17}

\newtheorem{theorem}{Theorem}

\newtheorem{lemma}[theorem]{Lemma}

\newenvironment{proof}[1][Proof]{\noindent\textbf{#1.} }{\ \rule{0.5em}{0.5em}}

\pgfplotstableset{
  /color cells/min/.initial=0,
  /color cells/max/.initial=1000,
  /color cells/textcolor/.initial=,
  color cells/.style={
    postproc cell content/.append code={%
          \pgfkeys{/color cells/.cd,#1}%
      \pgfkeysgetvalue{/pgfplots/table/@preprocessed cell content}\value%
      \ifx\value\empty%
      \else%
     
     \pgfmathparse{int(less(##1,-2))} 
       \ifnum\pgfmathresult=1
           \edef\temp{\noexpand\pgfkeyssetvalue{/pgfplots/table/@cell content}{%
            \noexpand\cellcolor{gray}{n/a}%
          }%
        }%
        \temp%
      \else%
      \pgfmathfloatparsenumber{\value}\pgfmathfloattofixed{\pgfmathresult}%
        \let\value=\pgfmathresult%
        \pgfplotscolormapaccess%
            [\pgfkeysvalueof{/color cells/min}:\pgfkeysvalueof{/color cells/max}]%
            {\value}{\pgfkeysvalueof{/pgfplots/colormap name}}%
        \pgfkeysgetvalue{/pgfplots/table/@cell content}\typesetvalue%
        \pgfkeysgetvalue{/color cells/textcolor}\textcolorvalue%
        \toks0=\expandafter{\typesetvalue}%
        \edef\temp{\noexpand\pgfkeyssetvalue{/pgfplots/table/@cell content}{%
            \noexpand\cellcolor[rgb]{\pgfmathresult}%
            \noexpand\definecolor{mapped color}{rgb}{\pgfmathresult}%
            \ifx\textcolorvalue\empty\else\noexpand\color{\textcolorvalue}\fi%
            \the\toks0%
          }%
        }%
        \temp%
      \fi%
      \fi%
      }%
  }%
}%
\begin{document}

\title{Convergence Analysis of an Endemic Time Delay Model Using Dirac and Radon Measures}
\author[1]{Tin Nwe Aye}
\author[2]{Linus Carlsson}

\affil[1]{Department of Mathematics, Linn\'{e}universitetet, Växjö, Sweden}
\affil[2]{Division of Applied Mathematics, Mälardalen University, Box 883, 721 23, Västerås, Sweden}
\maketitle

\begin{abstract}
This article explores the convergence properties of an $SLIR^\text{T}R^\text{P}D$ endemic model, incorporating Dirac and Radon measures, alongside distributed delays to represent latency and temporary immunity. A class of delays is defined for both continuous and discrete endemic models using continuous integral kernels with compact support and discrete terms expressed through Dirac and Radon measures. Numerical results show that the continuous model can be approximated by a discrete lag endemic model. Furthermore, the simulation time for the numerical solution is significantly shorter than that for the exact solution.
\end{abstract}

\section{Introduction}
In the study of disease spread, simple models frequently employ systems of differential equations (see, for example, \cite{castillo1996mathematical,dwyer1997host,tilahun2020stochastic,trejos2022dynamics}). To better capture the dynamics of infection flow within populations, delay differential models have also been explored extensively (e.g., \cite{guglielmi2022delay,wei2008delayed,zhang2008analysis,zhu2020dynamics,arino2006time,al2016time,misra2013mathematical,disease2024}).

In this study, we investigate the solutions of retarded differential equations. Specifically, we analyze the $SLIR^\text{T}R^\text{P}D$ endemic model with distributed delays as presented in \cite{TL2024}, focusing on its convergence properties. The model incorporates a probability kernel density function with compact support on the positive real axis, satisfying the condition
\begin{equation*}
\int_{\mathbb{R_{+}}}\Phi(\rho)d\rho=1.
\end{equation*}
By considering the Dirac measure, we analyze the convergence of numerical solutions of the system with discrete delays to the exact solution of the $SLIR^\text{T}R^\text{P}D$ model, which is formulated as a retarded differential equation. For distributed time delays, we specifically address the latency period and temporary immunity as integrated within the $SLIR^\text{T}R^\text{P}D$ model through delay differential equations.
Several research articles have explored similar realistic phenomenas,
incorporating continuous or discrete time delays for both discrete and
continuous models \cite{guglielmi2022delay, zhu2020dynamics,wei2008delayed, zhang2008analysis, beretta1995global}.

As an extension of the $SLIR^\text{T}R^\text{P}D$ model, we incorporate generalized Radon measures into the framework. Control at infinity of the positive Radon measure enables the relaxation of the assumption of compact support for the probability kernel density function, allowing us to handle non-compactly supported kernel densities. In both cases, models with continuous integral kernel functions and discrete kernel functions are introduced to analyze the convergence properties of the numerical solution of the discrete model toward the exact solution of the continuous model. This study demonstrates that the numerical approximations with discrete kernel functions converge to their continuous counterparts, ensuring consistency between the discrete and continuous formulations.

The structure of this paper is as follows: In Chapter~\ref{sec:modelWithDirac}
, we prove the convergence of a family of numerical approximations of the $SLIR^\text{T}R^\text{P}D$ model. The main result is presented in Chapter~\ref{Sec:GeneralMeasureKernel}, where we extend the $SLIR^\text{T}R^\text{P}D$ model to include Radon measure kernels and prove that a family of discrete lag models converges to this extended version. Chapter~\ref{sec:numericalExample} provides a numerical example to demonstrate the convergence. Finally, we conclude the paper with a discussion and suggest potential future research directions.

\section{Modeling Endemic Dynamics with Continuous Delays}\label{sec:modelWithDirac}
The $SLIR^\text{T}R^\text{P}D$ endemic model is formulated by using delay differential equations as a compartmental model for comprehending the propagation of a disease within an unstructured population.  The endemic continuous time delay model is derived as
\begin{align}\label{model:mainSystemOfEquation}
\frac{dS}{dt} &  =-\beta(t)I(t)S(t)+p\gamma\int_{\mathbb{R_{+}}}I(t-\rho
)\Phi(\rho)d\lambda(\rho),\nonumber\\
\frac{dL}{dt} &  =\beta(t)I(t)S(t)-\int_{\mathbb{R_{+}}}\beta(t-\tau
)I(t-\tau)S(t-\tau)\Psi(\tau)d\lambda(\tau),\nonumber\\
\frac{dI}{dt} &  =\int_{\mathbb{R_{+}}}\beta(t-\tau)I(t-\tau)S(t-\tau
)\Psi(\tau)d\lambda(\tau)-\gamma I(t)-\mu I(t),\\
\frac{dR^{\text{T}}}{dt} &  =p\gamma I(t)-p\gamma\int_{\mathbb{R_{+}}}%
I(t-\rho)\Phi(\rho)d\lambda(\rho),\nonumber\\
\frac{dR^{\text{P}}}{dt} &  =(1-p)\gamma I(t),\nonumber\\
\frac{dD}{dt} &  =\mu I(t).\nonumber
\end{align}

In the article \cite{TL2024}, we provided an in-depth analysis of the dynamics models with distributed delays, specifically focusing on their structures. In this paper, we will concentrate on demonstrating how System \ref{model:mainSystemOfEquation} can be approximated by System \ref{model:discrete} as the number of summation terms increases, i.e. we prove that discrete delay solutions converge to the continuous delay solution in supremum norm. The endemic discrete time delay model is defined as
\begin{align}
\frac{dS_{j}}{dt} &  =-\beta(t)I_{j}(t)S_{j}(t)+p\gamma\sum_{i=1}^{j}%
\omega_{i}^{j}I_{j}(t-\rho_{i}^{j}),\nonumber\\
\frac{dL_{j}}{dt} &  =\beta(t)I_{j}(t)S_{j}(t)-\sum_{i=1}^{j}\varpi_{i}%
^{j}\beta(t-\tau_{i}^{j})I_{j}(t-\tau_{i}^{j})S_{j}(t-\tau_{i}^{j}),\nonumber\\
\frac{dI_{j}}{dt} &  =\sum_{i=1}^{j}\varpi_{i}^{j}\beta(t-\tau_{i}^{j}%
)I_{j}(t-\tau_{i}^{j})S_{j}(t-\tau_{i}^{j})-\gamma I_{j}(t)-\mu I_{j}%
(t),\label{model:discrete}\\
\frac{dR_{j}^{\text{T}}}{dt} &  =p\gamma I_{j}(t)-p\gamma\sum_{i=1}^{j}%
\omega_{i}^{j}I_{j}(t-\rho_{i}^{j}),\nonumber\\
\frac{dR_{j}^{\text{P}}}{dt} &  =(1-p)\gamma I_{j}(t),\nonumber\\
\frac{dD_{j}}{dt} &  =\mu I_{j}(t).\nonumber
\end{align}
\subsection{Convergence analysis for delay differential equations}\label{sec:convergence}
System
\ref{model:mainSystemOfEquation} is assumed to satisfy that the contact rate
$\beta$ is a non-negative smooth function on $\mathbb{R}$, the history data for the susceptible and infected individuals are
constant, i.e.,%
\begin{align}
S(s)  &  =c_{S}>0,\text{ for all }s\leq0,\label{eq:HS}\\
I(s)  &  =c_{I}>0,\text{ for all }s\leq0, \label{eq:HI}%
\end{align}
and finally that the initial data for the latency, temporary recovered, and permanent recovered individuals are%
\begin{align}
L(0)  &  =\beta_{0}c_{I}c_{S}\int_{\theta}^{L}\Psi(\tau)\tau d\tau
\text{,}\label{eq:Es}\\
R^{\text{T}}(0)  &  =c_{I}p\gamma\int_{\sigma}^{M}\Phi(\rho)\rho
d\rho\text{,}\label{eq:Rs}\\
R^{\text{P}}(0)  &  =(1-p)\gamma c_{I}\int_{\theta}^{L}\Psi(\tau)\tau
d\tau\text{.} \label{eq:RPs}%
\end{align}
as derived in
\cite{TL2024} where
for some $0<\theta<L<\infty$ and $0<\sigma<M<\infty$, $\Psi(\tau)=0$ and $\Phi(\rho)=0$  
if $\tau\notin\lbrack\theta,L]$ and $\rho\notin\lbrack\sigma,M]$,
i.e.  $\operatorname*{supp}(\Psi)\subset\lbrack\theta,L]$ and $\operatorname*{supp}(\Phi)\subset\lbrack\sigma,M].$
 
In System \ref{model:discrete}, the history data are assumed to satisfy equations~\eqref{eq:Es}--\eqref{eq:RPs}, with $\sum_{i=1}^{j}\omega_{i}^j\rho_i^j$ and $\sum_{i=1}^{j}\varpi_{i}^j\tau_i^j$ in place of $\int_{\sigma}^{M}\Phi(\rho)\rho
d\rho$ and $\int_{\theta}^{L}\Psi(\tau)\tau
d\tau$.
 
Theorem 6 in article \cite{TL2024}, proves that these history/initial data gives 
solutions that are $L-$increasingly smooth, in particular, the solution is
continuous on $[0,\infty)$ and differentiable on $(0,\infty)$.  

To be able to show convergence, we chose the discrete delays in System~\ref{model:discrete} we use the time steps
\begin{equation}\label{eq:division}
d_{i}^{j}=\sigma+\frac{i(M-\sigma)}{j}
\end{equation}
where $i=0,1,2,...,j$, and $j>1$.

Additionally, since we will apply \emph{the first mean value theorem for definite integrals}, we select any
\begin{equation}\label{eq:Def_rho_i}
\rho_{i}^{j}\in\lbrack d_{i-1}^{j},d_{i}^{j}],i=1,2,...,j
\end{equation}
and define
\begin{equation}
\omega_{i}^{j}=\int_{d_{i-1}^{j}}^{d_{i}^{j}}\Phi(\rho) d\lambda(\rho)\label{eq:weighted}%
\end{equation}
An analogous construction is made for $\{\omega_{i}^{j}\}$  and $\{\varpi_{i}^{j}\}$. \newline
Finally, let ${d\nu_{j}(\rho)=\sum_{i=1}^{j}\omega
_{i}^{j}\delta_{\rho_{i}^{j}}(\rho)}$, where $\delta_{\rho_{i}^{j}}$ is the Dirac
measure centered in $\rho_{i}^{j}$. We use this discretization to be able to transist between the continuous integral setting and the discrete setting of sums of delays. In what follows, we will denote the continuous functions as $C(\mathbb{R})$. 

\begin{lemma}
\label{lem:approach} Let $H$ be uniformly continuous function and let $\Phi$ be a probability density function with $\mathrm{supp}(\Phi)\subset([\sigma,M])$ then%
\[
\int_{\mathbb{R_{+}}}H(\rho)d\nu_{j}(\rho)=\sum_{i=1}^{j}H(\rho_{i}^{j})\omega_{i}^{j}\underset{j\rightarrow\infty
}{\longrightarrow}\int_{\mathbb{R_{+}}}H(\rho)\Phi(\rho)d\lambda(\rho)
\]
In measure theoritical arguments, we prove that $d\nu_{j}\underset
{j\rightarrow\infty}{\longrightarrow}\Phi d\lambda$ on $C(\mathbb{R}_{+}).$
\end{lemma}

\begin{proof}
Fix any given $\varepsilon>0$. By using Equation \eqref{eq:weighted}, we have%
\begin{equation}\label{eq:eqH}
H(\rho_{i}^{j})\omega_{i}^{j}=H(\rho_{i}^{j})\int_{d_{i-1}^{j}}^{d_{i}^{j}%
}\Phi d\lambda=\int_{d_{i-1}^{j}}^{d_{i}^{j}}H(\rho_{i}^{j})\Phi d\lambda
\end{equation}

Fix $N>0$ large enough, so that for any $j>N$, the intervals
$I_{i}^{j}=[d_{i-1}^{j},d_{i}^{j}]$ are so small that the
uniformly continuous function $H$ satisfies
\begin{equation} \label{eq:epsilonError}
|H(\rho)-H(\rho_{i}^{j})|<\varepsilon
\end{equation}

or in other words%
\begin{equation}\label{eq:uniconFun}
H(\rho_{i}^{j})=H(\rho)+\sigma_{i}^{j}(\rho)
\end{equation}
where $|\sigma_{i}^{j}(\rho)|<\varepsilon$. \newline Note that since $\Phi$
is non-negative, we have%
\[
\left\vert \int_{d_{i-1}^{j}}^{d_{i}^{j}}\sigma_{i}^{j}(\rho)\Phi
(\rho)d\lambda(\rho)\right\vert \leq\varepsilon\int_{d_{i-1}^{j}}^{d_{i}^{j}%
}\Phi(\rho)d\lambda(\rho)
\]
Substituting Equation \eqref{eq:uniconFun} into Equation \eqref{eq:eqH}, we get%
\begin{align*}
H(\rho_{i}^{j})\omega_{i}^{j}  & =\int_{d_{i-1}^{j}}^{d_{i}^{j}}\left(
H(\rho)+\sigma_{i}^{j}(\rho)\right)  \Phi d\lambda(\rho)\\
& =\int_{d_{i-1}^{j}}^{d_{i}^{j}}H(\rho)\Phi(\rho)d\lambda(\rho)+\int
_{d_{i-1}^{j}}^{d_{i}^{j}}\sigma_{i}^{j}(\rho)\Phi(\rho)d\lambda(\rho)
\end{align*}
\newline Summing over $i$, gives%
\[
\sum_{i=1}^{j}H(\rho_{i}^{j})\omega_{i}^{j}=\sum_{i=1}^{j}\int_{d_{0}^{j}%
}^{d_{j}^{j}}H(\rho)\Phi(\rho)d\lambda(\rho)+\sum_{i=1}^{j}\int_{d_{i-1}^{j}%
}^{d_{i}^{j}}\sigma_{i}^{j}(\rho)\Phi(\rho)d\lambda(\rho)
\]
Again note that%
\begin{align*}
\left\vert \sum_{i=1}^{j}\int_{d_{i-1}^{j}}^{d_{i}^{j}}\sigma_{i}^{j}%
(\rho)\Phi(\rho)d\lambda(\rho)\right\vert  &  \leq\sum_{i=1}^{j}\left\vert
\int_{d_{i-1}^{j}}^{d_{i}^{j}}\sigma_{i}^{j}(\rho)\Phi(\rho)d\lambda
(\rho)\right\vert \\
&  \leq\sum_{i=1}^{j}\int_{d_{i-1}^{j}}^{d_{i}^{j}}\left\vert \sigma_{i}%
^{j}(\rho)\right\vert \Phi(\rho)d\lambda(\rho)\\
&  \leq\sum_{i=1}^{j}\int_{d_{i-1}^{j}}^{d_{i}^{j}}\varepsilon\Phi
(\rho)d\lambda(\rho)\\
&  \leq\varepsilon\sum_{i=1}^{j}\int_{d_{i-1}^{j}}^{d_{i}^{j}}\Phi
(\rho)d\lambda(\rho)\\
&  =\varepsilon\int_{\sigma}^{M}\Phi(\rho)d\lambda(\rho)\\
&  =\varepsilon
\end{align*}
We get
\begin{equation}\label{eq:epsilonLemma}
\left\vert \int_{\sigma}^{M}H(\rho)\Phi(\rho)d\lambda(\rho)
- \sum_{i=1}^{j}H(\rho_{i}^{j})\omega_{i}^{j}\right\vert \le \varepsilon
\end{equation}
And since $\varepsilon>0$ was arbitrary, it follows
\begin{equation}
\sum_{i=1}^{j}H(\rho_{i}^{j})\omega_{i}^{j}\underset{j\rightarrow\infty
}{\longrightarrow}\sum_{i=1}^{j}\int_{d_{0}^{j}}^{d_{j}^{j}}H(\rho)\Phi
(\rho)d\lambda(\rho)=\int_{\sigma}^{M}H(\rho)\Phi(\rho)d\lambda
(\rho)\label{eq:ApproForLarge}%
\end{equation}
Trivial concepts of integration theory helps us rewrite the left-hand side as%
\[
\sum_{i=1}^{j}H(\rho_{i}^{j})\omega_{i}^{j}=\sum_{i=1}^{j}\int_{\mathbb{R_{+}%
}} H(\rho)\omega_{i}^{j}\delta_{\rho_{i}^{j}}=\int_{\mathbb{R_{+}}}H(\rho
)\sum_{i=1}^{j}\omega_{i}^{j}\delta_{\rho_{i}^{j}}=\int_{\mathbb{R_{+}}}%
H(\rho)d\nu_{j}(\rho)
\]
Which finally shows that%
\[
\int_{\mathbb{R_{+}}}H(\rho)d\nu_{j}(\rho)\underset{j\rightarrow\infty
}{\longrightarrow}\int_{\mathbb{R_{+}}}H(\rho)\Phi(\rho)d\lambda(\rho)
\]

\end{proof}

\begin{theorem}\label{the:convergenceDirac}
Let $\{S,L,I,R^{\text{T}},R^{\text{P}}\}$ be the solution of System
\ref{model:mainSystemOfEquation}. The contact rate is assumed to be a
non-negative, smooth function and the remaining parameters are assumed to be
non-negative constants. Then, for any $T>0$, the solution to System \ref{model:discrete} with
$\{\omega_{i}^{j}\},\{\rho_{i}^{j}\},\{\varpi_{i}^{j}\},$ and $\{\tau_{i}%
^{j}\}$ converges in the supremum norm on $[0,T]$ to the solution of System \ref{model:mainSystemOfEquation}
as $j$ tends to infinity.
\end{theorem}
\begin{proof}
Let $T>0$ and $\varepsilon>0$ be given. \\
From Theorem 6 in \cite{TL2024}, we get that 
${S,I}\in C(\mathbb{R})$ and $I,R^{\text{T}},R^{\text{P}}\in C(\mathbb{R_+}).$
By Heine-Cantor theorem,  the functions 
${S,L,I,R^{\text{T}},R^{\text{P}}}$ are uniformly continuous on 
the compact sets $[-\max\{M,L\},T]$ and $[0,T]$ respectively. 
Fix any $0\leq t\leq T$. We want to
prove that $|I_{j}(t)-I(t)|<\varepsilon$ for $j>N_{\varepsilon}=\max\{N_\varepsilon^1,\cdots,N_\varepsilon^m\}$ where $m$ is the smallest integer such that $m\theta>T$ (remember that $\textrm{supp}(\psi)\subset[\theta,L]$). By using
System \ref{model:mainSystemOfEquation} and System \ref{model:discrete} we
get
\begin{align*}
\frac{d}{dt}\left(  I_{j}-I\right)  (t) &  =\left(  \sum_{i=1}^{j}\omega
_{i}^{j}\beta(t-\tau_{i}^{j})I_{j}(t-\tau_{i}^{j})S_{j}(t-\tau_{i}^{j})-\gamma
I_{j}(t)-\mu I_{j}(t)\right)  \\
&  -\int_{\theta}^L\beta(t-\tau)I(t-\tau)S(t-\tau)\Psi(\tau)d\tau-\gamma
I(t)-\mu I(t)\\
&  =\sum_{i=1}^{j}\omega_{i}^{j}\beta(t-\tau_{i}^{j})I_{j}(t-\tau_{i}%
^{j})S_{j}(t-\tau_{i}^{j})\\
&\quad -\int_{\theta}^L\beta(t-\tau)I(t-\tau
)S(t-\tau)\Psi(\tau)d\tau\\
&\quad  -(\gamma+\mu)\left(  I_{j}(t)-I(t)\right)
\end{align*}
Simplying,
\begin{align*}
\frac{d}{dt}\left(  I_{j}-I\right)  +(\gamma+\mu)\left(  I_{j}-I\right)
&=\sum_{i=1}^{j}\omega_{i}^{j}\beta(t-\tau_{i}^{j})I_{j}(t-\tau_{i}^{j}%
)S_{j}(t-\tau_{i}^{j})\\
&\quad -\int_{\theta}^L\beta(t-\tau
)I(t-\tau)S(t-\tau)\Psi(\tau)d\tau
\end{align*}
\newline 
By multiplying with the integrating factor for both sides, we get%
\begin{align*}
&\frac{d}{dt}\left(  \left(  I_{j}-I\right)  e^{t(\gamma+\mu)}\right)   \\
&=e^{t(\gamma+\mu)}\cdot\\
&\quad  \left(  \underbrace{\sum_{i=1}^{j}\omega_{i}^{j}\beta(r-\tau_{i}^{j}%
)I_{j}(r-\tau_{i}^{j})S_{j}(r-\tau_{i}^{j})-\int_{\theta}^{L}\beta
(r-\tau)I(r-\tau)S(r-\tau)\Psi(\tau)d\tau}_{=D_{j}(r,\tau)}\right)
\end{align*}
By integrating both side, and remembering that $I(0)=I_j(0)$, we get%

\begin{equation}\label{eq:converg_I}
I_{j}-I   =e^{-t\left(  \gamma+\mu\right)  }
\int_{0}^{t}   D_{j}(r,\tau) e^{r\left(\gamma+\mu\right)  }  dr
\end{equation}

Since $\mathrm{supp}(\Psi)=[\theta,L]$, we only consider the functions $S(t)$,  $I(t)$ and $\beta(t)$ on the closed and bounded interval, $[t-L,t-\theta]$,  in the above equation. 

To simplify the notations below, we set $H(\tau)=\beta(r-\tau)I(r-\tau)S(r-\tau)$ and 
$H_j(\tau^j_i)=\beta(r-\tau^j_i)I_j(r-\tau^j_i)S_j(r-\tau^j_i)$.

To prove that $|I(t)-I_j(t)|<\varepsilon$ for all $0 \le t \le T$ we will divide the interval up into subintervals as follows.
On the first time interval we will prove that $|I(t)-I_j(t)|<c_1\cdot \varepsilon$  for all $j>N_\varepsilon^1$.  
\newline i)
\textbf{Convergence on }$[0,\theta]$\textbf{ of }$I$: 

 For any $0\le t \le \theta$, we prove that there is an $N_\varepsilon^1$ 
such that $|I(t)-I_j(t)|<c_1\cdot \varepsilon$  for all $j>N_\varepsilon^1$.

In section \ref{sec:convergence},  we mention that the history data 
of $S(s)$ and $I(s)$ are constant for all $s\leq0$. That make $H(\tau)$ and $H_j(\tau^j_i)$ to be the same according to the assumption of the constant history data since $r-\tau\leq 0$ for all $0\le r \le t \le \theta$.  With that setting,  we can use
the Equation \eqref{eq:ApproForLarge} in Lemma \ref{lem:approach}
with $\Psi$ in place of $\Phi$ and $\tau$ in place of $\rho$, that (remember, from above $H=H_j$)
\[
\sum_{i=1}^{j}H(\tau_{i}^{j})\omega_{i}^{j}\underset{j\rightarrow\infty
}{\longrightarrow}\int_{\theta}^LH(\tau)\Psi(\tau)d\lambda(\tau)
\]

It follows that  there exists an $N_\varepsilon^1>0$ such that
$|D_{j}(r,\tau)|<\varepsilon$ if $j>N_{\varepsilon}^1$ for all $0\le r \le \theta$, due to the uniform continuity.

As a consequence, we get
\begin{align*}
|I_{j}(t)-I(t)| &  <e^{-t\left(  \gamma+\mu\right)  }\int_{0}^{t}e^{r\left(
\gamma+\mu\right)  }\varepsilon dr\\
&  =\frac{\varepsilon}{(\gamma+\mu)}e^{-t\left(  \gamma+\mu\right)  }\left(
e^{t\left(  \gamma+\mu\right)  }-1\right)  \\
&  <c_1\varepsilon
\end{align*}
for $0\le t \le \theta$.
\newline ii)
\textbf{Convergence on }$[0,\theta]$\textbf{ of }$S$:
 
For $0\le t \le \theta$,  we will prove that $\left\vert S_{j}(t)-S\right(t)\vert <\varepsilon$ for
any $\varepsilon>0$ and $j>N_{\varepsilon}^1.$%
\begin{align*}
\frac{d}{dt}\left(  S_{j}-S\right)   &  =\left(  -\beta(t)I_{j}%
(t)S_{j}(t)+p\gamma\sum_{i=1}^{j}\omega_{i}^{j}I_{j}(t-\rho_{i}^{j})\right)
\\
&  -\left(  -\beta(t)I(t)S(t)+p\gamma\int_{\sigma}^MI(t-\rho
)\Phi(\rho)d\rho\right)
\end{align*}
Simplifying%
\begin{align*}
\frac{d}{dt}\left(  S_{j}-S\right)   &  =p\gamma\left(  \sum_{i=1}%
^{j}\omega_{i}^{j}I_{j}(t-\rho_{i}^{j})-\int_{\sigma}^M I(t-\rho
)\Phi(\rho)d\rho\right)  \\
&  -\beta(t)\left(  I_{j}(t)S_{j}(t)-I(t)S(t)\right)  \\
&  =p\gamma\left(  \sum_{i=1}^{j}\omega_{i}^{j}I_{j}(t-\rho_{i}^{j}%
)-\int_{\sigma}^MI(t-\rho
)\Phi(\rho)d\rho\right)  \\
&  -\beta(t)\left(  I_{j}(t)S_{j}(t)
\underbrace{-I_{j}(t)S(t)+I_{j}(t)S(t)}_{\textrm{add zero}}%
-I(t)S(t)\right)  \\
&  =p\gamma\left(  \sum_{i=1}^{j}\omega_{i}^{j}I_{j}(t-\rho_{i}^{j}%
)-\int_{\sigma}^MI(t-\rho
)\Phi(\rho)d\rho\right)  \\
&  -\beta(t)I_{j}(t)\left(  S_{j}(t)-S(t)\right)  -\beta(t)S(t)\left(
I_{j}(t)-I(t)\right)
\end{align*}
By re-arranging the terms, we obtain
\begin{align*}
\frac{d}{dt}\left(  S_{j}-S\right)  +\beta I_{j}\left(  S_{j}%
-S\right)   &  =p\gamma\left(  \sum_{i=1}^{j}\omega_{i}^{j}I_{j}%
(t-\rho_{i}^{j})-\int_{\sigma}^MI(t-\rho
)\Phi(\rho)d\rho\right)  \\
&  -\beta S\left(  I_{j}-I\right)
\end{align*}
\newline 
By multiplying both sides of the equation by the integrating factors, we get%
\begin{align*}
\frac{d}{dt}\left(  \left(  S_{j}-S\right)  \exp(
{\displaystyle\int\limits_{0}^{t}}
\beta(s)I_j(s)ds)\right)   &  =p\gamma\exp(%
{\displaystyle\int\limits_{0}^{t}}
\beta(s)I_j(s)ds)\cdot\\
&\quad \left(  \sum_{i=1}^{j}\omega_{i}^{j}I_{j}(t-\rho_{i}%
^{j})-\int_{\sigma}^MI(t-\rho
)\Phi(\rho)d\rho\right)  \\
&\quad  -\exp(%
{\displaystyle\int\limits_{0}^{t}}
\beta(s)I_j(s)ds)\cdot\beta S\left(  I_{j}-I\right)
\end{align*}
which solves to%
\begin{align}\label{eq:converg_S}
S_{j}-S&  =p\gamma\exp(-%
{\displaystyle\int\limits_{0}^{t}}
\beta(s)I_j(s)ds)\cdot\nonumber\\
&
{\displaystyle\int\limits_{0}^{t}}
\exp(%
{\displaystyle\int\limits_{0}^{r}}
\beta(s)I_j(s)ds)\cdot\left(  \underbrace{\sum_{i=1}^{j}\omega_{i}^{j}I_{j}(r-\rho_{i}%
^{j})-\int_{\sigma}^MI(r-\rho
)\Phi(\rho)d\rho}_{=D_{j}(r,\rho)}\right)  dr\nonumber\\
&  -\exp(-%
{\displaystyle\int\limits_{0}^{t}}
\beta(s)I_j(s)ds)\cdot%
{\displaystyle\int\limits_{0}^{t}}
\exp(%
{\displaystyle\int\limits_{0}^{r}}
\beta(s)I_j(s)ds)\cdot\beta S\left(  \underbrace{I_{j}-I}_{<c_1\varepsilon}\right)  dr
\end{align}
By using the fact $\vert I_{j}(t)-I(t)\vert<c_1\varepsilon$ for $0\leq t\leq \theta$ if $j>N_\varepsilon^1$ and using 
Equation~\eqref{eq:ApproForLarge}, 
in Lemma \ref{lem:approach} proves that
$\vert D_{j}(r,\rho)\vert<\varepsilon$ for $0\leq r\leq t\leq \theta$. This holds because 
$r-\rho\leq t-\rho\leq \theta-\rho\leq \theta $. 
Note
that $%
{\displaystyle\int\limits_{0}^{r}}
\beta(s)I_j(s)ds$ is uniformly bounded since $\beta$ and  $I$ are bounded and continuous function. Therefore it
follows that
\begin{align*}
\vert S_{j}-S\vert&  <p\gamma\varepsilon \exp(-%
{\displaystyle\int\limits_{0}^{t}}
\beta(s)I_j(s)ds)\cdot
{\displaystyle\int\limits_{0}^{t}}
\exp(%
{\displaystyle\int\limits_{0}^{r}}
\beta(s)I_j(s)ds) dr\\
&  -c_1\varepsilon \exp(-%
{\displaystyle\int\limits_{0}^{t}}
\beta(s)I_j(s)ds)\cdot%
{\displaystyle\int\limits_{0}^{t}}
\exp(%
{\displaystyle\int\limits_{0}^{r}}
\beta(s)I_j(s)ds)\cdot\beta S dr\\
&=\varepsilon \exp(-%
{\displaystyle\int\limits_{0}^{t}}
\beta(s)I_j(s)ds)\cdot\\%
&\quad \left(p\gamma{\displaystyle\int\limits_{0}^{t}}\exp(
{\displaystyle\int\limits_{0}^{r}}
\beta(s)I_j(s)ds) dr -c_1{\displaystyle\int\limits_{0}^{t}}\exp(
{\displaystyle\int\limits_{0}^{r}}
\beta(s)I_j(s)ds) \cdot\beta(r)S(r)dr \right)\\
&<c_2\varepsilon 
\end{align*}
for $0\leq t\leq \theta$. 
\newline iii)
\textbf{Convergence on }$[\theta,2\theta]$\textbf{ of }$I$: 

 For $\theta\le t \le 2\theta$,  we prove that there is an $N_\varepsilon^2$ 
such that $|I(t)-I_j(t)|<c_3\cdot \varepsilon$  for all $j>N_\varepsilon^2$. 

From Equation \eqref{eq:converg_I},  we have
\begin{align*}
&D_{j}(r,\tau)\\
&=\sum_{i=1}^{j}\omega_{i}^{j}\beta(r-\tau_{i}^{j}%
)I_{j}(r-\tau_{i}^{j})S_{j}(r-\tau_{i}^{j})-\int_{\theta}^L\beta
(r-\tau)I(r-\tau)S(r-\tau)\Psi(\tau)d\tau\\
&=\underbrace{\sum_{i=1}^{j_1}\omega_{i}^{j}\beta(r-\tau_{i}^{j}%
)I_{j}(r-\tau_{i}^{j})S_{j}(r-\tau_{i}^{j})-\int_{\theta}^{2\theta}\beta
(r-\tau)I(r-\tau)S(r-\tau)\Psi(\tau)d\tau}_{g_1^j(r,\tau)}\\
&+\underbrace{\sum_{i=j_1+1}^{j}\omega_{i}^{j}\beta(r-\tau_{i}^{j}%
)I_{j}(r-\tau_{i}^{j})S_{j}(r-\tau_{i}^{j})-\int_{2\theta}^{L}\beta
(r-\tau)I(r-\tau)S(r-\tau)\Psi(\tau)d\tau}_{g_2^j(r,\tau)}
\end{align*}
In the above summation term,  we chose $j_1<j$ such that the interval, divided into two parts, satisfies  ${d_{j_1}^{j}<2\theta}$
 and ${d_{j_1+1}^j\ge2\theta}$.

By using the Equation \eqref{eq:ApproForLarge} in Lemma \ref{lem:approach}, we have $|g_2^j(r,\tau)|<\varepsilon$ for $0\leq r\leq t$ and $2\theta\leq \tau\leq L$ if $j>N_\varepsilon^2$
since $S(r-\tau)$ and $I(r-\tau)$ are constant for $r-\tau\leq 0$.

By using $\vert I_{j}(t)-I(t)\vert<c_1\varepsilon $ and $\vert S_{j}(t)-S(t)\vert<c_2\varepsilon $ for $0\leq t \leq \theta$ if $j>N_\varepsilon^1$,  
we can say that $\vert I_{j}(r-\tau)-I(r-\tau)\vert<c_1\varepsilon $ and $\vert S_{j}(r-\tau)-S(r-\tau)\vert<c_2\varepsilon $ since $0\leq r-\tau\leq \theta$.
Then,  
\begin{align*}
\vert \beta S_{j}I_{j}-\beta SI\vert&=\beta\vert S_{j}I_{j}-SI\vert\\
&=\beta\vert S_{j}I_{j}-S_{j}I+S_{j_1}I-SI\vert\\
&\leq \beta(\vert S_{j}I_{j}-S_{j}I\vert+\vert S_{j}I-SI\vert)\\
&= \beta (S_{j}\vert I_{j}-I\vert+I \vert S_{j}-S\vert)\\
&\leq \beta \varepsilon (c_1S_{j}+ c_2I)
\end{align*}
Since $S_{j_1}$ and $I$ are bounded functions, we have
$$\vert \beta S_{j}I_{j}-\beta SI\vert<c_3\varepsilon$$

So, we can use  the Equation \eqref{eq:ApproForLarge} in Lemma \ref{lem:approach} since $\vert \beta S_{j}I_{j}-\beta SI\vert<c_3\varepsilon$. As a result, we have $\vert g_1^j(r,\tau)\vert<c_4\varepsilon$ if $j>N_\varepsilon^2$.

Hence, we know that $\vert D_{j}(r,\tau)\vert<c_5\varepsilon$ if $j>\max\{N_\varepsilon^1,N_\varepsilon^2\}$. 
As a consequence, we get from Equation \eqref{eq:converg_I}
\begin{align*}
|I_{j}(t)-I(t)| &  <e^{-t\left(  \gamma+\mu\right)  }\int_{0}^{t}e^{r\left(
\gamma+\mu\right)  }c_5\varepsilon dr\\
&  =\frac{c_5\varepsilon}{(\gamma+\mu)}e^{-t\left(  \gamma+\mu\right)  }\left(
e^{t\left(  \gamma+\mu\right)  }-1\right)  \\
&  <c_6\varepsilon
\end{align*}
for $\theta\le t \le 2\theta$.
\newline iv)
\textbf{Convergence on }$[\theta,2\theta]$\textbf{ of }$S$: 

By using that $|I_{j}(t)-I(t)| <c_6\varepsilon$, for $\theta\le t \le 2\theta$, we get 
$\vert D_{j}(r,\rho)\vert<\varepsilon$ if $j>\max\{N_\varepsilon^1,N_\varepsilon^2\}$ according to the Equation \eqref{eq:ApproForLarge} in Lemma \ref{lem:approach}.

From Equation \eqref{eq:converg_S}, we get
\begin{align*}
\vert S_{j}-S\vert&  <p\gamma\varepsilon \exp(-%
{\displaystyle\int\limits_{0}^{t}}
\beta(s)I_j(s)ds)\cdot
{\displaystyle\int\limits_{0}^{t}}
\exp(%
{\displaystyle\int\limits_{0}^{r}}
\beta(s)I_j(s)ds) dr\\
&  -c_6\varepsilon \exp(-%
{\displaystyle\int\limits_{0}^{t}}
\beta(s)I_j(s)ds)\cdot%
{\displaystyle\int\limits_{0}^{t}}
\exp(%
{\displaystyle\int\limits_{0}^{r}}
\beta(s)I_j(s)ds)\cdot\beta S dr\\
&<c_7\varepsilon 
\end{align*}
for $\theta\leq t\leq 2\theta$. 
\newline v)
\textbf{Convergence on }$[0,T]$\textbf{ of }$I$: 
 
In the above calculations, we have a constant multiplied with $\varepsilon$ in our inequalities, 
but, e.g., for $S$ above, we have that $c_2$ and $c_7$ does not depend on $j$ nor $\varepsilon$, therefore, 
we may pick a value  $\mathcal{N} > \max\{N_\varepsilon^1,N_\varepsilon^2\}$ 
such that  $$\vert S_{j}(t)-S(t)\vert<\varepsilon$$ for all $t\in[0,2\theta]$ as long as $j>\mathcal{N}$. An analogous argument shows the same inequality for $I$ that is, $$\vert I_j(t)-I(t) \vert <\varepsilon$$ on the same time interval.

We prove that $|I(t)-I_j(t)|<\varepsilon$ for all $0 \le t \le T$
by means of finite mathematical induction.  

Fix $m$ as the smallest integer such that $m\theta>T$.

Base case: Above, we have proved that $|I(t)-I_j(t)|<\varepsilon$ and $|S(t)-S_j(t)|<\varepsilon$ for all $0 \le t \le 2\theta$ if $j>\max\{N_\varepsilon^1,N_\varepsilon^2\}$.

Asumption: We assume that $|I(t)-I_j(t)|<\varepsilon$ and $|S(t)-S_j(t)|<\varepsilon$,  for all $0 \le t \le k\theta$, is true, for a fixed integer $1 \le k \le m-1$ if $j>\max\{N_\varepsilon^1,N_\varepsilon^2,\ldots,N_\varepsilon^k\}$.

Induction step: We will
now prove that $|I(t)-I_j(t)|<\varepsilon$ for all $0 \le t \le (k+1)\theta$ is true if $j>\max\{N_\varepsilon^1,N_\varepsilon^2,\ldots,N_\varepsilon^{k+1}\}$.

We will have two cases,  $\theta \le t \le L\le (k+1)\theta$ and $\theta\le t\le (k+1)\theta\le L$.

\textbf{Case}: $0 \le t \le L\le (k+1)\theta$.

By using $\vert I_{j}(t)-I(t)\vert<\varepsilon $ and $\vert S_{j}(t)-S(t)\vert<\varepsilon $ for $0\leq t \leq k\theta$,  
we can say that $\vert I_{j_2}(r-\tau)-I(r-\tau)\vert<\varepsilon $ and $\vert S_{j_2}(r-\tau)-S(r-\tau)\vert<\varepsilon $ since $0 \le r-\tau \le t-\tau\le (k+1)\theta-\theta=k\theta$ for all $0\le r\le t$.
Then,  
\begin{align}\label{eq:SIEq}
\vert \beta S_{j_1}I_{j_1}-\beta SI\vert&=\beta\vert S_{j_1}I_{j_1}-SI\vert\nonumber\\
&=\beta\vert S_{j_1}I_{j_1}-S_{j_1}I+S_{j_1}I-SI\vert\nonumber\\
&\leq \beta(\vert S_{j_1}I_{j_1}-S_{j_1}I\vert+\vert S_{j_1}I-SI\vert)\nonumber\\
&= \beta (S_{j_1}\vert I_{j_1}-I\vert+I \vert S_{j_1}-S\vert)\nonumber\\
&\leq \beta \varepsilon (S_{j_1}+ I)
\end{align}
Since $S_{j_1}$ and $I$ are bounded functions, we have
$$\vert \beta S_{j_1}I_{j_1}-\beta SI\vert<c\varepsilon$$

As a consequence, we may use Equation \eqref{eq:ApproForLarge} in Lemma \ref{lem:approach}, to show that 
${\vert D_j(r,\tau)\vert <\varepsilon}$ in Equation \eqref{eq:converg_I}.

For $0\le t\le L\le (k+1)\theta$ we get 
\begin{equation*}
|I(t)-I_j(t)|<\varepsilon e^{-t\left(  \gamma+\mu\right)  }\int_{0}^{t}e^{r\left(
\gamma+\mu\right)  }dr=\frac{\varepsilon}{r+\mu}\left( 1-e^{-t(r+\mu)}\right)=c_8\varepsilon.
\end{equation*}

\textbf{Case}: $0\le t\le (k+1)\theta\le L$.

Again, Equation \eqref{eq:converg_I} provides 
\begin{align*}
&D_{j}(r,\tau)\\
&=\sum_{i=1}^{j}\omega_{i}^{j}\beta(r-\tau_{i}^{j}%
)I_{j}(r-\tau_{i}^{j})S_{j}(r-\tau_{i}^{j})-\int_{\theta}^L\beta
(r-\tau)I(r-\tau)S(r-\tau)\Psi(\tau)d\tau\\
&=\underbrace{\sum_{i=1}^{j_2}\omega_{i}^{j}\beta(r-\tau_{i}^{j}%
)I_{j}(r-\tau_{i}^{j})S_{j}(r-\tau_{i}^{j})-\int_{\theta}^{(k+1)\theta}\beta
(r-\tau)I(r-\tau)S(r-\tau)\Psi(\tau)d\tau}_{g_3^j(r,\tau)}\\
&+\underbrace{\sum_{i=j_2+1}^{j}\omega_{i}^{j}\beta(r-\tau_{i}^{j}%
)I_{j}(r-\tau_{i}^{j})S_{j}(r-\tau_{i}^{j})-\int_{(k+1)\theta}^{L}\beta
(r-\tau)I(r-\tau)S(r-\tau)\Psi(\tau)d\tau}_{g_4^j(r,\tau)}
\end{align*}
In the above summation terms, $j_2$, is chosen such that,  
${d_{j_2}^{j}<(k+1)\theta}$ and 
${d_{j_2+1}^j\ge (k+1)\theta}$.

By using the Equation \eqref{eq:ApproForLarge} in Lemma \ref{lem:approach},  $\vert g_4^j(r,\tau)\vert<\varepsilon$ for $0\le r\le t$ and $(k+1)\theta\le \tau\le L$
since $S(r-\tau)$,  $S_j(r-\tau)$,  $I(r-\tau)$ and  $I_j(r-\tau)$ are the constants for $r-\tau\leq 0$.

Then according to the Equation \eqref{eq:ApproForLarge} in Lemma \ref{lem:approach} we have $\vert g_3^j(r,\tau)\vert<\varepsilon$ 
if ${j>\max\{N_\varepsilon^1,N_\varepsilon^2,\cdots,N_\varepsilon^k\}}$ since the Equation \eqref{eq:SIEq} is valid by using the assumption step.
Thus, it follows that $\vert D_{j}(r,\tau)\vert<c_8\varepsilon$ for $0\le r\le t$ for all 
${j>\max\{N_\varepsilon^1,N_\varepsilon^2,\cdots,N_\varepsilon^{k+1}\}}$.
As a consequence, we get from Equation \eqref{eq:converg_I}
\begin{align*}
|I_{j}(t)-I(t)| &  <e^{-t\left(  \gamma+\mu\right)  }\int_{0}^{t}e^{r\left(
\gamma+\mu\right)  }\varepsilon dr\\
&  =\frac{c_5\varepsilon}{(\gamma+\mu)}e^{-t\left(  \gamma+\mu\right)  }\left(
e^{t\left(  \gamma+\mu\right)  }-1\right)  \\
&  <c_9\varepsilon
\end{align*}
for $0\le t \le (k+1)\theta$.
By Mathematical induction,  we get that $|I_{j}(t)-I(t)| <\varepsilon$ for $0\le t \le T$ as long as $j>\max\{N_\varepsilon^1,N_\varepsilon^2,\cdots,N_\varepsilon^m\}$.
\newline vi)
\textbf{Convergence on }$[0,T]$ \textbf{ of } $S$:

By analoguos arguments as iv), one can prove the case of convergence on $[0,T]$ of $S$ by using the convergence of $I$ on $[0,T]$.
\newline vii) 
\textbf{Final note}:

The reason why we may exclude the constants in the equality above stems from the same arguments as in the beginning of this proof, and since this is repeated a finite amount of times, there will be a a finite constant $\mathcal{N}$ such that inequalities
$$\vert S_{j}(t)-S(t)\vert<\varepsilon,\quad \vert I_j(t)-I(t) \vert <\varepsilon$$ for all $t\in[0,T]$ as long as $j>\mathcal{N}$.
\end{proof}

\section{Endemic model with generalized measure approach}\label{Sec:GeneralMeasureKernel}

In the continuous and discrete endemic models from Section~\ref{sec:modelWithDirac}, we assumed that the probability density functions $\Phi$ and $\Psi$ have compact support in $L^{1}(\mathbb{R})$, using both Lebesgue and Dirac measures. 
From here on, we extend 
System~\eqref{model:mainSystemOfEquation} and System~\eqref{model:discrete} 
to include kernal functions with unbounded support, that is, 
\begin{align}\label{eq:FactsWRadon}
\operatorname*{supp}(\Phi)&\subset [\sigma,\infty), \textrm{where } \sigma>0, \nonumber\\ 
\operatorname*{supp}(\Psi)&\subset [\theta,\infty), \textrm{where } \theta > 0,\\ 
\int_{\mathbb{R_{+}}} \Phi(\rho) d\lambda(\rho)& = 1=\int_{\mathbb{R_{+}}} \Psi(\tau) d\lambda(\tau)\nonumber. 
\end{align}
Here the non-negative measures ${d\mu_1 = \Psi d\lambda}$ and ${d\mu_2 = \Phi d\lambda}$ are viewed as Radon measures. 
Specifically, for any $\varepsilon > 0$, there exists $M$ such that ${\int_{M}^{\infty} \Phi(\rho)  d\lambda(\rho) < \varepsilon}$.

\subsection{Convergence analysis of endemic model with generalized measure}\label{sec:convergenceGenMea}
In this section,  we analyze the convergence of the endemic model using a generalized Radon measure, rather than relying on the assumption of compact support with a Dirac measure. The important reason for working with positive Radon measures is that their behaviour at infinity is tightly controlled, that is for each $\epsilon>0$, there exists a compact set $K_\epsilon$ such that $\mu(\mathbb{R_{+}}\setminus K_\epsilon)<\epsilon$.

The endemic discrete time model is defined as System \eqref{model:discrete} and the history data are assumed as mentioned in Section \ref{sec:convergence}.

\begin{lemma} \label{lem:approachRadon}
Let $H:\mathbb{R_+}\rightarrow \mathbb{R_+}$ be continuous and bounded function such that $H d\lambda$ becomes a positive finite Radon measure, then%
\[
\int_{\mathbb{R_{+}}}H(\rho)d\nu_{j}(\rho)=\sum_{i=1}^{j}H(\rho_{i}^{j})\omega_{i}^{j}\underset{j\rightarrow\infty
}{\longrightarrow}\int_{\mathbb{R_{+}}}H(\rho)\Phi(\rho)d\lambda(\rho)
\]
where $\lambda$ is the Lebesgue measure.
\end{lemma}

\begin{proof}
Fix any given $\varepsilon>0$. Using the boundedness of $H$ and the tightness of the Radon measures, we get an $M>0$ such that
\begin{equation}\label{eq:RadonControl}
\int_{M}^\infty H(\rho)\Phi d\lambda<\varepsilon/2
\end{equation}
With this $M$ we use Equation~\eqref{eq:division} to define the interval end points $d_i^j$. Now we employ Lemma 1 on the interval $[\sigma,M]$ and use Equation~\eqref{eq:epsilonLemma} (which is true, even if $\int_{\sigma}^M \Phi<1$) to get
\begin{equation*}
\left\vert \int_{\sigma}^{M}H(\rho)\Phi(\rho)d\lambda(\rho)
- \sum_{i=1}^{j}H(\rho_{i}^{j})\omega_{i}^{j}\right\vert \le \varepsilon/2
\end{equation*}
We now have
\begin{align*}
\int_{\mathbb{R_{+}}}&H(\rho)\Phi(\rho)d\lambda(\rho)-\sum_{i=1}^{j}H(\rho_{i}^{j})\omega_{i}^{j}\\
&=\int_{\sigma}^M H(\rho)\Phi(\rho)d\lambda(\rho)+
\left(\int_{M}^{\infty} H(\rho)\Phi(\rho)d\lambda(\rho)-\sum_{i=1}^{j}H(\rho_{i}^{j})\omega_{i}^{j}\right)
\end{align*}

Applying the modulus function on both sides and using the triangle inequality ends the proof.
\end{proof}

\begin{theorem}
Let $\{S,L,I,R^{\text{T}},R^{\text{P}}\}$ be the solution of 
System~\eqref{model:mainSystemOfEquation} with the support of complementary Equation \eqref{eq:FactsWRadon}.
The contact rate is assumed to be a
non-negative, smooth function and the remaining parameters are assumed to be
non-negative constants. Then, for any $T>0$, the solution of System~\eqref{model:discrete} with
$\{\omega_{i}^{j}\},\{\rho_{i}^{j}\},\{\varpi_{i}^{j}\},$ and $\{\tau_{i}%
^{j}\}$ converges in the supremum norm on $[0,T]$ to the solution of System~\eqref{model:mainSystemOfEquation}
as $j$ tends to infinity.
\end{theorem}

\begin{proof}
The proof can be demonstrated through an analougous construction as used in Theorem \ref{the:convergenceDirac} by using Lemma \ref{lem:approachRadon}.
\end{proof}

\section{Numerical example}\label{sec:numericalExample}

In this section, we introduce a discrete time delay endemic version of 
System~\eqref{model:mainSystemOfEquation} with the support of complementary Equation \eqref{eq:FactsWRadon}. 
An investigation of how the solution
of this discrete time delay model approximate the exact solution of 
System~\eqref{model:mainSystemOfEquation}.  
The exact solution and numerical approximation algorithms are described below. 
We set the disease death rate to $\mu=\gamma I_{\text{FR}} (1-I_{\text{FR}})^{-1}$,   
see \cite{disease2024} for details. 

In the following numerical approximations we simulate a disese with shortest latency time $\theta=5$ days, and shortest time of immunity $\sigma=10$ days, and the remaining parameters are presented in Table~\ref{tableVar}.

\begin{table}[h]
\centering%
\begin{tabular}
[c]{|llll|}\hline
Symbol & Value & Unit & Interpretation\\\hline
$\gamma$ & 0.1$(=1/\hat{\gamma})$ & $\text{day}^{-1}$ & recovery rate
(1/duration of sickness)\\\hline
$N$ & $10^7$ & $-$ & initial popultion size\\\hline
$\beta$ & 0.5/N & $\text{day}^{-1}$ & contact rate\\\hline
$I_{FR}$ & 0.425 & $\text{day}^{-1}$ & infection fatality risk\\\hline
$p$ & 0.9 & $-$ & proportional immunity parameter\\\hline
$\tau$ & 1-100 & $\text{day}^{-1}$ & latent time\\\hline
$\rho$ & 2-200 & $\text{day}^{-1}$ & duration of temporary immunity\\\hline
$\mu$ & 0.0739 & $\text{day}^{-1}$ & disease death rate\\\hline
\end{tabular}
\caption{Description of model parameters. }%
\label{tableVar}%
\end{table}
\newpage

\subsection{Discrete finite time-delay approximation using DDEs}\label{sec:discreteRadon}
In this section, we present the discrete time-lag algorithm used to approximate the solution of System~\eqref{model:mainSystemOfEquation}
with the support of complementary Equation \eqref{eq:FactsWRadon} using the DDE23 solver. Both continuous and discrete integral kernels are defined for this simulation. Specifically, we utilize the following exponential distribution functions:

\begin{equation}
\Phi(\rho)=\left\{
\begin{array}
[c]{cc}%
\lambda_1 e^{-\lambda_1 (\rho-\sigma)}& \rho \ge \sigma,\\
0 & \text{elsewhere.}%
\end{array}
\right.  \label{eq:Rho}%
\end{equation}
and
\begin{equation}
\Psi(\tau)=\left\{
\begin{array}
[c]{cc}%
\lambda_2 e^{-\lambda_2 (\tau-\theta)}& \tau\ge \theta,\\
0 & \text{elsewhere.}%
\end{array}
\right.  \label{eq:Tau}%
\end{equation}

Assigning the intitial population to 10 millions, wich is an upper bound for all comparments, we may use $H=10^9$ in 
Equation \eqref{eq:RadonControl}, to determine the value of $M$ where we have used $\sigma = 10$, $\theta = 5$, 
and to make sure that we have probability kernels, 
we use $\lambda_1 = 1/\sigma$, and $\lambda_2 = 1/\theta$. 
This results in $M = 86$ (see the detailed explanation in Appendix \ref{app:Calculation}). 
The value of $M$ allows us to disregard the interval from $M$ to infinity, 
as the solutions in this interval—using exponential distribution functions—become negligible. 
This makes it feasible to apply the discrete endemic model with the aforementioned kernel functions.

In the discrete endemic model described by System~\eqref{model:discrete}, 
the integration terms in System~\eqref{model:mainSystemOfEquation} are replaced by summation terms, 
requiring the division of the entire integration interval.
We begin by dividing these intervals as outlined in equations~\eqref{eq:division}--\eqref{eq:weighted}. 
Here, $N_{\tau}$ denotes the number of subintervals used to approximate the finite integrals with respect to $\tau$, 
while $N_{\rho}$ is defined analogously for integrals with respect to $\rho$. 
This approach, used for simulating the discrete lag endemic model,
is referred to as the \emph{discrete $(N_{\tau}, N_{\rho})$ model for multiple lags}, as described in Section 4.2 of \cite{TL2024}.

\subsection{Continuous time-delay approximation using ODEs}

In our endemic model described by System~\eqref{model:mainSystemOfEquation}, with with the support of complementary Equation \eqref{eq:FactsWRadon} and the simplified kernel functions defined by equations~\eqref{eq:Rho}--\eqref{eq:Tau}, 
we can explicitly integrate the generalized integrals. This allows us to solve the system on consecutive time intervals using ordinary differential equations (ODEs), rather than delay differential equations.

Below, we rewrite System~\eqref{model:mainSystemOfEquation} and apply the ODE solver \texttt{ODE45} to obtain the ``exact'' solution. To do so, we follow the algorithm presented in Section 4.3 of \cite{TL2024}, 
with the modification of incorporating the exponential distribution functions described earlier. 

Finally, we compare this solution with the numerical results of the discrete-lag endemic model presented in System~\eqref{model:discrete} and discussed in Section~\ref{sec:discreteRadon}.

We rewrite System~\eqref{model:mainSystemOfEquation} on the form
\begin{align}\label{model:mainSystemOfEquationExact}
\frac{dS}{dt} &  =-\beta(t)I(t)S(t)+p\gamma G(t),\nonumber\\
\frac{dL}{dt} &  =\beta(t)I(t)S(t)-H(t),\nonumber\\
\frac{dI}{dt} &  =H(t)-\gamma I(t)-\mu I(t),\\
\frac{dR^{\text{T}}}{dt} &  =p\gamma I(t)-p\gamma G(t),\nonumber\\
\frac{dR^{\text{P}}}{dt} &  =(1-p)\gamma I(t),\nonumber\\
\frac{dD}{dt} &  =\mu I(t),\nonumber
\end{align}
where we emphasize that this is an ordinary system of ODEs and the initial data is defined above in this section.  The functions $G(t)$ and $H(t)$ are described as
\begin{equation}\label{eq:G}
G(t)=\int_{\sigma}^\infty I(t-\rho)\Phi(\rho)d\lambda(\rho)
\end{equation}
and 
\begin{equation}\label{eq:H}
H(t)=\int_{\theta}^\infty \beta(t-\tau)I(t-\tau)S(t-\tau
)\Psi(\tau)d\lambda(\tau).
\end{equation}

By using the equations \eqref{eq:Rho} and \eqref{eq:Tau}, the equations \eqref{eq:G} and \eqref{eq:H} can be presented as follows:

If $t\ge \sigma$ we get
\begin{align*}
G(t)&=\int_{\sigma}^\infty I(t-\rho
)\Phi(\rho)d\lambda(\rho)\\
&= \int_{\sigma}^\infty I(t-\rho
)\Phi(\rho)d\lambda(\rho)\\
&= \int_{\sigma}^t I(t-\rho
)\Phi(\rho)d\lambda(\rho)+\int_{t}^\infty I(t-\rho
)\Phi(\rho)d\lambda(\rho)\\
&= \int_{\sigma}^t I(t-\rho
)\Phi(\rho)d\lambda(\rho)+c_I\int_{t}^\infty \Phi(\rho)d\lambda(\rho)
\end{align*}
and if $t<\sigma$ the initial condition is $G(t)=c_I$.
By using the Equation \eqref{eq:Rho},  we get
\begin{align*}
G(t)&=
\int_{\sigma}^t I(t-\rho
)\lambda_1 e^{-\lambda_1 (\rho-\sigma)}d\lambda(\rho)+c_I\int_{t}^\infty \lambda_1 e^{-\lambda_1 (\rho-\sigma)}d\lambda(\rho)\\
&=\int_{\sigma}^t I(t-\rho
)\lambda_1 e^{-\lambda_1 (\rho-\sigma)}d\lambda(\rho)+c_I e^{-\lambda_1(t-\sigma)}
\end{align*}
In a similar way,  if $t<\theta$ then $H(t)=\beta_0 c_I c_s$ otherwise ($t\ge \theta$ ) we can solve
\begin{align*}
H(t)&=\int_{\theta}^\infty \beta(t-\tau)I(t-\tau)S(t-\tau
)\Psi(\tau)d\lambda(\tau)\\
&=\int_{\theta}^t \beta(t-\tau)I(t-\tau)S(t-\tau
)\Psi(\tau)d\lambda(\tau)\\
&+\beta_0 c_I c_S\int_{t}^\infty \Psi(\tau)d\lambda(\tau)\\
\end{align*}
We evaluate the last term by the use of Equation \eqref{eq:Tau},  we get
\begin{align*}
H(t)&=\int_{\theta}^t \beta(t-\tau)I(t-\tau)S(t-\tau
)\lambda_2 e^{-\lambda_2 (\tau-\theta)} d\lambda(\tau)\\
&+\beta_0 c_I c_S e^{-\lambda_2(t-\theta)}
\end{align*}
In addition, to determine the initial conditions, i.e. equations \eqref{eq:Es}--\eqref{eq:RPs}, we use the kernal functions defined by equations \eqref{eq:Rho} and \eqref{eq:Tau}.

By substituting the Equation \eqref{eq:Tau} into Equation \eqref{eq:Es} and using integration by parts gives
\[
L(0)   =\beta_{0}c_{I}c_{S} \left(\theta+\frac{1}{\lambda_2}\right)
\]
Similar calulations gives
\begin{align*}
R^{\text{T}}(0)  &  =c_{I}p\gamma \left(\sigma+\frac{1}{\lambda_1}\right)
\text{and}\\
R^{\text{P}}(0)  &  =(1-p)\gamma c_{I} \left(\theta+\frac{1}{\lambda_2}\right)
\text{.}%
\end{align*}

The initial condition of $S(0)$, is calculated as
\begin{align*}
S(0)=N(0)-L(0)-I(0)-R^{\text{T}}(0)-R^{\text{P}}(0)
\end{align*}
By using the above initial conditions, we get
\begin{align*}
S(0)&=N(0)-\beta_{0}c_{I}S(0) \left(\theta+\frac{1}{\lambda_2}\right)-c_I\\
&\quad-c_{I}p\gamma \left(\sigma+\frac{1}{\lambda_1}\right)-(1-p)\gamma c_{I} \left(\theta+\frac{1}{\lambda_2}\right)
\end{align*}
Simplifying, we get
\begin{equation*}
S(0)=\frac{N(0)-c_I-c_{I}p\gamma \left(\sigma+\frac{1}{\lambda_1}\right)-(1-p)\gamma c_{I} \left(\theta+\frac{1}{\lambda_2}\right)}{\left(1+ \beta_{0}c_{I}\left(\theta+\frac{1}{\lambda_2}\right)\right)}
\end{equation*}
The algorithm from \cite{TL2024} is employed to obtain the ``exact solution'' using the ODE45 solver in MATLAB. In the next section, 
we use this exact solution as a benchmark to compare with the numerical solution of System \eqref{model:discrete}, 
which corresponds to the discrete $(N_{\tau}, N_{\rho})$ model described in Section \ref{sec:discreteRadon}.

\subsection{Simulation results}\label{sec:simulations}

In this section, we compare the exact solutions of System \eqref{model:mainSystemOfEquationExact} 
with the numerical solutions of System~\eqref{model:mainSystemOfEquation} to demonstrate that the numerical solutions converge to the exact solution as $(N_{\tau},N_{\rho})$ increases. 
The parameter values are provided in Table \ref{tableVar}, and the kernels are given by equations \eqref{eq:Rho} and \eqref{eq:Tau}. 
The initial population for all compartments is computed using the above equations, with $N(0)$ and $I(0)$ set to 10 million and 10, respectively.

In the following illustration, we simulate the discrete lag endemic model by setting $(N_{\tau},N_{\rho})$ to $(1,2)$, $(10,20)$, and $(100,200)$, 
showing that increasing $(N_{\tau},N_{\rho})$ leads to convergence towards the exact solution of System \eqref{model:mainSystemOfEquationExact}.

\begin{figure}[H]
\centering
\includegraphics[width=5.8cm, height=5cm]{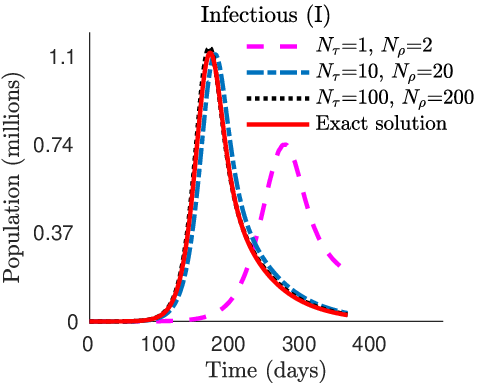}
\caption{The number of infectious individuals is presented 
under one year by setting different values of discrete ($N_{\tau}$,$N_{\rho}$) model and the continuous endemic model. 
The solid curve is shown by using the exact solution of System \eqref{model:mainSystemOfEquation} equipped with the kernal function equations~\eqref{eq:Rho}~and~\eqref{eq:Tau}. 
The rest of dotted curves are simulated utilizing discrete ($N_{\tau}$,$N_{\rho}$) model by using the values $(1,2)$,  $(10,20)$ and $(100,200)$.}\label{fig:infectious}
\end{figure}

It is evident from Figure~\ref{fig:infectious} that the solution of the discrete $(1,2)$ model deviates significantly from the exact solution curve,
whereas the other dotted curves are relatively close to the exact solution. 
We observe that the solutions of the discrete $(N_{\tau},N_{\rho})$ model converge toward the exact solution as $(N_{\tau},N_{\rho})$ increases, 
corroborating the analytical convergence results from Section~\ref{Sec:GeneralMeasureKernel}. 
In addition, the simulations of the number of individuals for other compartments are presented in the Appendix \ref{app:simulations}.
\begin{figure}[H]
\centering
\includegraphics[width=5.8cm, height=5cm]{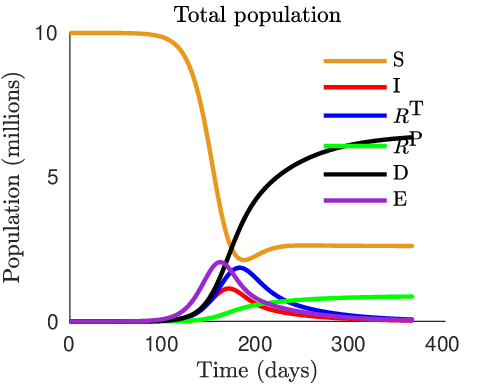}
\caption{The number of individuals for each compartment is presented 
under one year by setting $N_{\tau}=100$ and $N_{\rho}=200$ in the discrete ($N_{\tau}$,$N_{\rho}$) model, equipped with the kernal function equations~\eqref{eq:Rho}~and~\eqref{eq:Tau}. }\label{fig:allHundred}
\end{figure}

In Figure~\ref{fig:allHundred}, the total number of individuals in each compartment is shown as an approximation of the exact solution using the discrete $(100,200)$ model. The simulation time for this model is significantly faster compared to the continuous endemic model of System \ref{model:mainSystemOfEquationExact}. Specifically, the discrete model runs in approximately one minute, whereas the continuous model requires around four days of simulation time. Both simulations were run on a Windows laptop using MATLAB 2023b with the built-in functions ODE45 and DDE23.

\section{Conclusions and Remarks}
We propose a numerical method to find solutions to the $SLIR^\text{T}R^\text{P}D$ endemic model, which was introduced and analytically studied in \cite{TL2024}. In Section \ref{sec:modelWithDirac}, we establish the uniform convergence of the solutions obtained through this numerical method to the analytical solutions of the $SLIR^\text{T}R^\text{P}D$ model, as presented in this article in System \eqref{model:mainSystemOfEquation}. 

Moreover, we extend the $SLIR^\text{T}R^\text{P}D$ model by incorporating kernel functions with unbounded support and prove that the proposed numerical method converges uniformly on compact sets even in this more general setting. Our numerical approach approximates the Lebesgue/Radon-integral formulation by a finite sum of Dirac measures.

In the numerical simulation section, we present an example using exponential distribution functions as the probability density functions. 
This example illustrates the convergence of the numerical solutions of System \eqref{model:discrete} to the exact solution of 
System \eqref{model:mainSystemOfEquationExact}, 
with the corresponding conditions given by Equation~\eqref{eq:FactsWRadon}. 
The simulation results confirm that the discrete ($N_{\tau},N_{\rho}$) model of System \eqref{model:discrete} 
converges to the exact solution of System \eqref{model:mainSystemOfEquationExact} as $N_{\tau}$ and $N_{\rho}$ increase. 
Furthermore, it was observed that the computation time needed to find the exact solution is significantly longer—by several orders of magnitude—than the simulation time for the numerical approximation.

Another challenge arises when general probability kernel functions are employed, as integral in Inequality \eqref{eq:RadonControl} does not admit explicit solutions, in contrast to the case with the exponential distribution function.

As a direction for future work, one could explore the inclusion of general finite Radon measures as probability distributions.

%
%

\bibliographystyle{plain}
\bibliography{mybib}

\section*{Appendix}

\subsection*{Calculating $M$ in Section \ref{sec:discreteRadon}}\label{app:Calculation}
To approximate the value of $M$, we investigate the Equation \eqref{eq:RadonControl} by using the exponential distribution \eqref{eq:Rho}, we get
\begin{align*}
\int_M^\infty H\Phi(\rho)d\rho&=\int_M^\infty H\lambda_1 e^{-\lambda_1 (\rho-\sigma)}d\rho\\
&=He^{-\lambda_1(M-\sigma)}
\end{align*}
According to the Equation \eqref{eq:RadonControl},  the bound, given in Section \ref{app:Calculation}, implies that
\begin{equation*}
He^{-\lambda_1(M-\sigma)}\le 1
\end{equation*}
Thus,
\begin{equation*}
M\le \sigma+\frac{1}{\lambda_1}\ln(H)
\end{equation*}
In the simulation, we set $M=\sigma+\frac{1}{\lambda_1}\ln(H)$.
\subsection*{Numerical simulations}\label{app:simulations}
In this section, we present the rest of the simulation results of the discrete $(N_{\tau}, N_{\rho})$ model, 
presented in Section \ref{sec:simulations}, to demonstrate the convergence of the numerical solutions toward the exact solution of the continuous endemic model.

\begin{figure}[H]
\centering
\includegraphics[width=7.3cm, height=5cm]{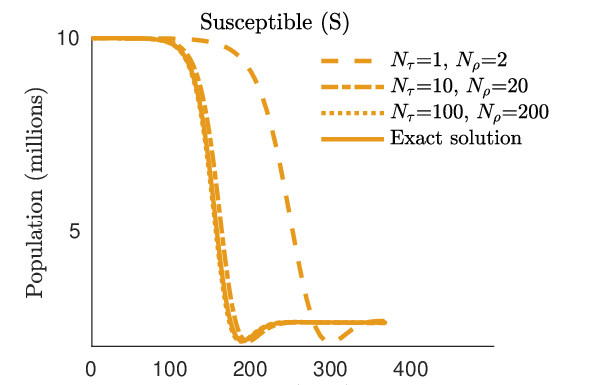}
\caption{The number of susceptible individuals is presented 
under 1 year by setting different values of discrete ($N_{\tau}$,$N_{\rho}$) model and the continuous endemic model. The solid curve is shown by using the exact solution of System \eqref{model:mainSystemOfEquation}. The rest of dotted curves are simulated utilizing discrete ($N_{\tau}$,$N_{\rho}$) model by setting (1,2),  (10,20) and (100,200).}\label{fig:susceptible}
\end{figure}

\begin{figure}[H]
\centering
\includegraphics[width=7.3cm, height=5cm]{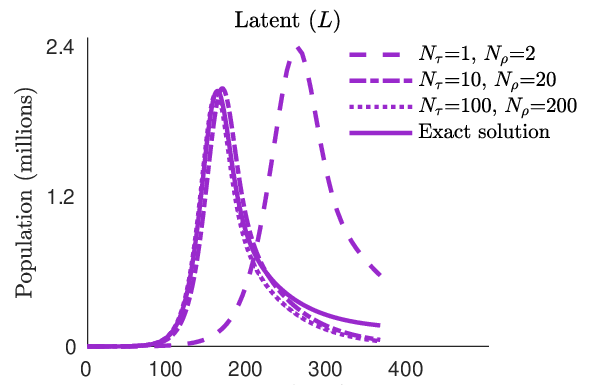}
\caption{The number of latent individuals is presented 
under 1 year by setting different values of discrete ($N_{\tau}$,$N_{\rho}$) model and the continuous endemic model. The solid curve is shown by using the exact solution of System \eqref{model:mainSystemOfEquation}. The rest of dotted curves are simulated utilizing discrete ($N_{\tau}$,$N_{\rho}$) model by setting (1,2),  (10,20) and (100,200).}\label{fig:L}
\end{figure}

\begin{figure}[H]
\centering
\includegraphics[width=7.3cm, height=5cm]{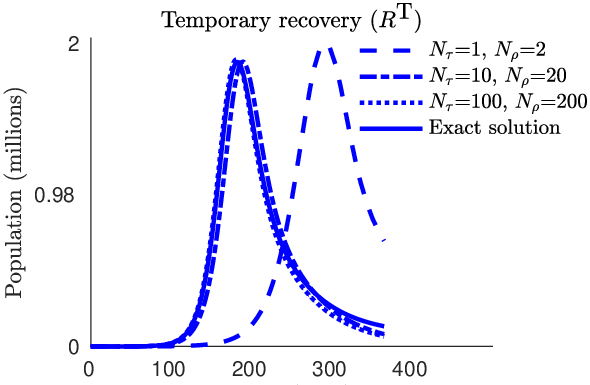}
\caption{The number of temporary recovery individuals is presented 
under 1 year by setting different values of discrete ($N_{\tau}$,$N_{\rho}$) model and the continuous endemic model. The solid curve is shown by using the exact solution of System \eqref{model:mainSystemOfEquation}. The rest of dotted curves are simulated utilizing discrete ($N_{\tau}$,$N_{\rho}$) model by setting (1,2),  (10,20) and (100,200).}\label{fig:temRecovery}
\end{figure}

\begin{figure}[H]
\centering
\includegraphics[width=7.3cm, height=5cm]{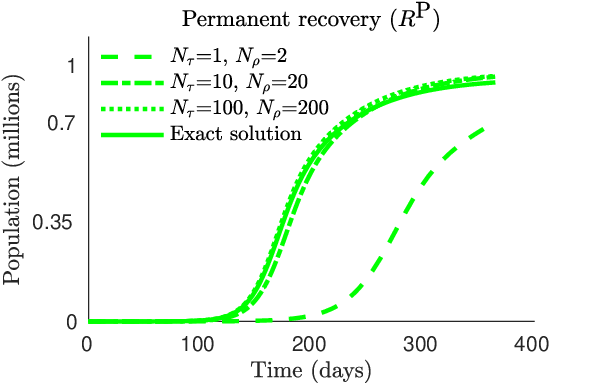}
\caption{The number of permanent recovery individuals is presented 
under 1 year by setting different values of discrete ($N_{\tau}$,$N_{\rho}$) model and the continuous endemic model. The solid curve is shown by using the exact solution of System \eqref{model:mainSystemOfEquation}. The rest of dotted curves are simulated utilizing discrete ($N_{\tau}$,$N_{\rho}$) model by setting (1,2),  (10,20) and (100,200).}\label{fig:perRecovery}
\end{figure}

\begin{figure}[H]
\centering
\includegraphics[width=7.3cm, height=5cm]{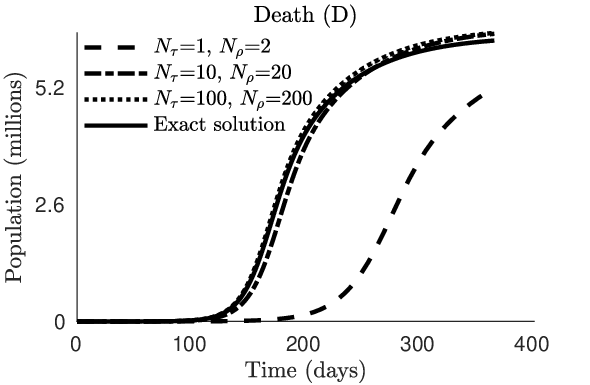}
\caption{The number of dead individuals is presented 
under 1 year by setting different values of discrete ($N_{\tau}$,$N_{\rho}$) model and the continuous endemic model. The solid curve is shown by using the exact solution of System \eqref{model:mainSystemOfEquation}. The rest of dotted curves are simulated utilizing discrete ($N_{\tau}$,$N_{\rho}$) model by setting (1,2),  (10,20) and (100,200).}\label{fig:dead}
\end{figure}

\end{document}